\documentclass[10pt,a4paper]{article}
\usepackage[final]{optional}

\usepackage{amsfonts, amsmath, wasysym}
\usepackage{enumerate}

\usepackage{tikz}

\usepackage{amssymb,amsthm,
paralist
}

\usepackage{
latexsym,
}

\usepackage{url}

\definecolor{darkgreen}{rgb}{0,0.5,0}
\definecolor{darkred}{rgb}{0.7,0,0}
\usepackage[colorlinks, 
citecolor=darkgreen, linkcolor=darkred
]{hyperref}

\theoremstyle{plain}
\newtheorem{lemma}{Lemma}[section]
\newtheorem{thm}[lemma]{Theorem}
\newtheorem{prop}[lemma]{Proposition}
\newtheorem{cor}[lemma]{Corollary}

\theoremstyle{definition}

\newtheorem{rmk}[lemma]{Remark}

\numberwithin{equation}{section}

\newcommand{\ck}{\ensuremath{{\mathcal K}}}

\newcommand{\pl}[2]{{\frac{\partial #1}{\partial #2}}}

\newcommand{\al}{\alpha}
\newcommand{\be}{\beta}
\newcommand{\ga}{\gamma}

\newcommand{\de}{\delta}

\newcommand{\ep}{\varepsilon}

\newcommand{\R}{\ensuremath{{\mathbb R}}}

\newcommand{\bx}{{\bf x}}
\newcommand{\by}{{\bf y}}

\newcommand{\downto}{\downarrow}
\newcommand{\upto}{\uparrow}

\newcommand{\lap}{\Delta}

\newcommand{\grad}{\nabla}

\DeclareMathOperator{\Vol}{Vol}

\DeclareMathOperator{\inj}{inj}

\newcommand{\norm}[1]{\left\Vert#1\right\Vert}  
\newcommand{\set}[1]{\left\{ #1 \right\}}

\def\blbox{\quad \vrule height7.5pt width4.17pt depth0pt}

\newcommand{\beq}{\begin{equation}}
\newcommand{\eeq}{\end{equation}}
\newcommand{\beqa}{\begin{equation}\begin{aligned}}
\newcommand{\eeqa}{\end{aligned}\end{equation}}
\newcommand{\brmk}{\begin{rmk}}
\newcommand{\ermk}{\end{rmk}}
\newcommand{\partref}[1]{\hbox{(\csname @roman\endcsname{\ref{#1}})}}
\newcommand{\half}{\frac{1}{2}}

\newcommand{\cmt}[1]{\opt{draft}{\textcolor[rgb]{0.5,0,0}{
$\LHD$ #1 $\RHD$\marginpar{\blbox}}}}

\newcommand{\Ric}{{\mathrm{Ric}}}

\newcommand{\Hess}{{\mathrm{Hess}}}

\newcommand{\pt}{\partial_t}
\newcommand{\abs}[1]{\left\vert#1\right\vert}

\title{{
\bf
Rate of curvature decay for the \\ contracting cusp Ricci flow
} 
\\ 
\cmt{DRAFT with comments}
}
\author{Peter M. Topping and Hao Yin}
\date{\today}

\begin{document}

\maketitle

\begin{abstract}
We prove that the Ricci flow that contracts a hyperbolic cusp has curvature decay $\max K\sim \frac{1}{t^2}$. In order to do this, we prove a new Li-Yau type differential Harnack inequality for Ricci flow.
\end{abstract}

\section{Introduction}
\label{intro}

Consider a Ricci flow $g(t)$ on a surface $M$, existing over a time interval $t\in (0,T)$, i.e. a smooth solution to 
\begin{equation}
\pl{g}{t}=-2\Ric_{g(t)}=-2Kg,
	\label{eqn:rf}
\end{equation}
where $K$ is the Gauss curvature.
As $t\upto T$, the curvature may blow up; for example if $g_S$ is the spherical metric on $S^2$ with $K\equiv 1$, then the solution $g(t)=(1-2t)g_S$ has curvature $K=(1-2t)^{-1}$.
Similarly, as $t\downto 0$, the curvature may also blow up. A first example of this would be if $g_H$ is a hyperbolic metric and $g(t)=2tg_H$, in which case the curvature is $K=-(2t)^{-1}$. A second example would be the natural Ricci flow smoothing out a conical surface, for which the supremum of $|K|$ also blows up like $Ct^{-1}$, with $C$ depending on the cone angle (see \cite{GHMS} or Section 4, Chapter 2 of \cite{ChowKnopf}). 

Although all of these most obvious examples have curvature blow-up like $(time)^{-1}$, Hamilton and Daskalopoulos \cite{HD} constructed examples for which the curvature blows up at time $T$ like 
$(time)^{-2}$. For example, if one takes the spherical metric on the \emph{punctured} sphere $S^2\backslash\{p\}$ with $K\equiv 1$, then the subsequent unique instantaneously complete Ricci flow
\cite{GT2, ICRF_UNIQ} will exist until time $T=1$, and the supremum of $K$ will blow up like $(1-t)^{-2}$ (up to some factor).

Later, in \cite{revcusp}, the first author constructed a new class of solutions that could be seen to have a  rate of blow-up as $t\downto 0$ that could not be $(time)^{-1}$. 
To understand one example of such a flow, consider the unique
complete hyperbolic conformal metric on $B\setminus \set{0}$, where $B\subset \R^2$ is
the unit disc, which has a hyperbolic cusp at the origin. This metric can be written 
$H=h(dx^2+dy^2)$ where
\begin{equation*}
	h(x,y)= \frac{1}{r^2(\log r)^2},
\end{equation*}
for $r=\sqrt{x^2+y^2}$, and the function $h$ restricts to an $L^1$ function on any compact subset of $B$ (i.e. the cusp has finite area). 
The arguments in \cite{revcusp} imply that other than the homothetically dilating solution $(1+2t)h$ defined on $B\setminus \set{0}$, there is an alternative complete `contracting cusp' Ricci flow solution $g_{cc}(t)$ defined on $B$, which caps the cusp off at infinity and then allows  it to contract.

\begin{thm}
\label{exist_unique}
There exists a smooth, complete Ricci flow 
$g_{cc}(t)=u_{cc}(t)(dx^2+dy^2)$ on $B$ for $t>0$ such that $g_{cc}(t)\to H$ smoothly locally on $B\backslash\{0\}$, 
and 
\beq
\label{gcc_L1}
u_{cc}(t)\to h\qquad\text{ in }L^1_{loc}(B),
\eeq
as $t\downto 0$.
Moreover, if $\tilde g(t)$ is any other smooth, complete Ricci flow on $B$, defined for 
$t\in (0,T)$, with $H$ as $L^1$ initial data in the sense of \eqref{gcc_L1}, then $g(t)=\tilde g(t)$ for all $t\in (0,T)$.
\end{thm}
The uniqueness assertion is a little different to that in \cite{revcusp}, and will be proved in Section \ref{unique_sect}. Moreover, the new $L^1-L^\infty$ smoothing estimate developed in \cite{TY} will allow us to give a streamlined proof of the existence, in Section \ref{sec:existence}, with some additional control.

It was conjectured in \cite{revcusp} that the curvature should blow up like $t^{-2}$, analogous to the result of Hamilton and Daskalopoulos, and it is this conjecture that we settle in this paper.  

\begin{thm}
\label{mainthm}
For some universal number $c_2>0$ and any $c_1>32$, if $K_{cc}$ is the Gauss curvature of $g_{cc}$, then 
$$\frac{1}{c_1 t^2}\leq \max_B K_{cc}(t)\leq \frac{c_2}{t^2}$$
for sufficiently small $t>0$ depending only on $c_1$.
\end{thm}
Although rather standard, as we discuss in Section \ref{not_bad_at_infinity}, it is worth recording that the curvature blow up above is happening asymptotically at the origin in the following sense.
\begin{prop}
\label{not_bad_at_infinity_prop}
For each $\ep\in (0,1)$, we have $K_{cc}(t)\to -1$ uniformly on $B\backslash B_\ep$ as $t\downto 0$.
\end{prop}
In contrast, the techniques of \cite{GT4} would allow one to construct Ricci flows with different rates of curvature blow-up on noncompact surfaces where the blow-up occurs not locally but at spatial infinity. Wu constructed  higher dimensional Ricci flows with other blow-up rates in \cite{Wu}.

The proof of Theorem \ref{mainthm} will use a combination of techniques, the main ones not exploiting the rotational symmetry of our setting, and the result could be generalised. 
One ingredient will be sharp estimates on the conformal factor of the flow.
Recall that \eqref{eqn:rf} is equivalent to
\begin{equation}
	\partial_t u = \triangle \log u
	\label{eqn:rfu}
\end{equation}
if $g(t)= u(t)(dx^2+dy^2)$, or equivalently
\begin{equation}
	\partial_t v = e^{-2v} \triangle v =-K
	\label{eqn:rfv}
\end{equation}
if $u(t)=e^{2v(t)}$.
The new $L^1-L^\infty$ smoothing estimate from \cite{TY} will give  sharp upper control on the conformal factor of $g_{cc}(t)$ almost immediately.
Moreover, by comparing with a family of cigar solutions, we will also obtain a sharp lower bound of the conformal factor, which is also an improvement of the lower bounds in \cite{revcusp}, established using Perelman's pseudolocality theorem. More precisely, we will prove
\begin{thm}
	\label{secondthm} Let $u_{cc}$ be the conformal factor of $g_{cc}$ on $B$ and $v_{cc}:=\frac{1}{2}\log u_{cc}$. 
Then 
\begin{equation*}
\frac{1}{8t}+\half\left(1+\log(4t)\right)\leq v_{cc}(0,t)
\end{equation*}
for $t\in (0,1/4)$, and 
\begin{equation*}
\max_{\bx \in B_{1/2}}v_{cc}(\bx, t) \leq \frac{1}{t}+C
\end{equation*}
for $t\in (0,1)$ and universal $C<\infty$.
\end{thm}
In other words, the conformal factor at the origin decays at a rate between $1/(8t)$ and $1/t$, neglecting lower order terms.

The upper bounds on the conformal factor will be exploited by a new Li-Yau differential Harnack estimate (Theorem \ref{thm:liyau}) that we prove in Section \ref{sec:liyau}. This estimate is more reminiscent of the original Li-Yau estimates \cite{LiYauActa} than the curvature Harnack estimates of Hamilton for Ricci flow \cite{surface}, although we will exploit crucially that we are working on a Ricci flow solution.
The upshot of that estimate will be that we can control the curvature from above in terms of the supremum of the conformal factor divided by time. Then the upper bound of $v_{cc}$ in Theorem \ref{secondthm} can be converted into the upper bound of curvature in Theorem \ref{mainthm}.  

Meanwhile, the lower bound for the conformal factor at the origin will be key in order to obtain the lower bound for the curvature:
The solution must remain under the hyperbolic cusp solution $(1+2t)h$ for all time
(see Lemma \ref{max_stretch_lem} and its consequence Remark \ref{new_sandwich_rmk} below), 
so the large conformal factor at the origin (for small $t$) implies some large bending near the origin, which gives the desired lower curvature bound in Theorem \ref{mainthm}. The details of this argument appear in Section \ref{sec:lower}.

Finally, we remark that an alternative way of deriving sharp estimates for the flow $g_{cc}(t)$ would be the method of rigorous matched asymptotic expansions. This alternative approach could give slightly refined asymptotic information, but would be harder to generalise to the non-rotationally symmetric case.

\section{The contracting cusp Ricci flow solution}
\label{sec:addon}
The first aim
of this section is to construct the Ricci flow solution $g_{cc}$, which has the hyperbolic cusp metric as the initial data in the 
sense of Theorem \ref{exist_unique} and caps off the cusp at  infinity instantaneously. 
(An underlying aim is also to gather the sharp estimates that will be required to prove our curvature asymptotics.)
Giesen and the first author 
established the well-posedness of instantaneously complete Ricci flows from any {\it smooth} initial data whether complete or not \cite{GT2,ICRF_UNIQ}. 
The hyperbolic cusp metric can be regarded as a metric defined on $B$ with a singular point at the origin. A natural approach for constructing a Ricci flow solution from it is to consider a sequence of smooth metrics approximating the hyperbolic cusp metric and to show that the sequence of Ricci flow solutions starting from these approximations converges to the desired solution.

\subsection{An approximation of the hyperbolic cusp}
\label{sec:approx}
There are many different ways of choosing the approximation sequence here. It will be clear by the end of Section \ref{unique_sect} that the final limit solution will be independent of the choice.
However, for later use in Section \ref{sec:liyau}, we shall use a special sequence obtained by considering cigar metrics touching the hyperbolic cusp.

Recall that the standard cigar metric on $\R^2$ is given by 
\begin{equation*}
	\frac{1}{1+ r^2} (dx^2+dy^2).
\end{equation*}
We introduce two parameters $\varepsilon>0$ and $\delta>0$ and consider the family of conformal factors
\begin{equation*}
	\frac{\varepsilon}{\delta+ r^2}
\end{equation*}
of scaled cigars. 
Note that $\varepsilon$ determines the maximum curvature of the metric and $\delta$ is irrelevant to the geometry and is related to the parametrization, or equivalently, since the cigar is a steady soliton, is related to time.

Amongst this two-parameter family of cigar metrics, we are interested in a one-parameter subfamily of metrics that are tangent to the hyperbolic cusp. 
For each $\de>0$, we can increase $\ep$ from zero until the conformal factors first touch.
It is obvious from Figure \ref{fig:cigar} or a simple argument that the cigar will only be tangent to the cusp on a circle 
$r=r_0$, with $r_0<\frac{1}{e}$. Moreover, as indicated in Figure \ref{fig:env}, the family of cigar metrics has the cusp metric as an envelope up to the `horizon' at $r=e^{-1}$.

\newcommand*{\sampno}{50} 

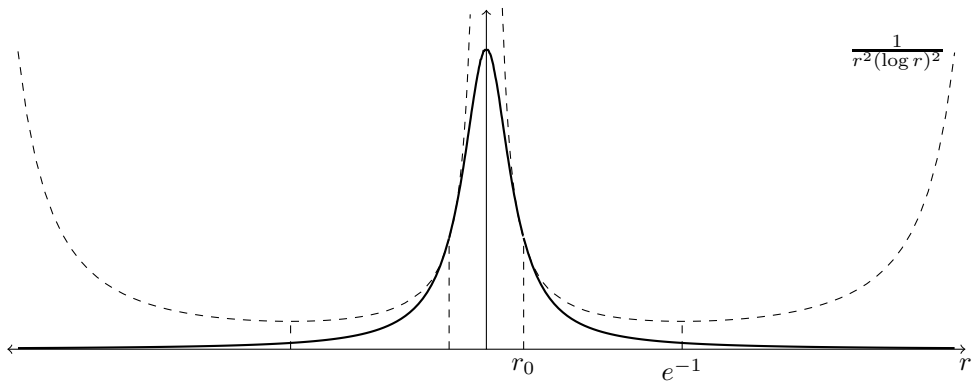
\begin{figure}
\centering
\begin{tikzpicture}[xscale=7,yscale=0.05]

\newcommand*{\rlimit}{0.88}
\newcommand*{\axisheight}{90}
\newcommand*{\rzero}{0.07}
\newcommand*{\mlogrzero}{-ln(\rzero)}
\newcommand*{\epsi}{1/(\mlogrzero*(\mlogrzero-1))}
\newcommand*{\delt}{\rzero*\rzero/(\mlogrzero-1)}
\newcommand*{\ee}{2.71828}

\draw [<->] (0,\axisheight) -- (0,0) -- (0.9,0) node[below]{$r$} ;
\draw [->] (0,0) -- (-0.9,0);

\draw[dashed, domain=\rzero:\rlimit, samples=\sampno] 
plot (\x, {
1/(\x * ln(\x))^2
}) node[left]{$\frac{1}{r^2(\log r)^2}$};

\draw[dashed, domain=0.03:\rzero, samples=\sampno] 
plot (\x, {
1/(\x * ln(\x))^2
}); 

\draw[dashed, domain=-\rlimit:-\rzero, samples=\sampno] 
plot (\x, {
1/(\x * ln(-\x))^2
});

\draw[dashed, thin, domain=-\rzero:-0.03, samples=\sampno] 
plot (\x, {
1/(\x * ln(-\x))^2
}); 

\draw[thick, domain=-\rzero:\rzero, samples=2*\sampno] 
plot (\x, {
\epsi/(\delt+(\x * \x))
});
\draw[thick, domain=-\rlimit:-\rzero, samples=2*\sampno] 
plot (\x, {
\epsi/(\delt+(\x * \x))
});
\draw[thick, domain=\rzero:\rlimit, samples=2*\sampno] 
plot (\x, {
\epsi/(\delt+(\x * \x))
});

\draw[dashed, thin] (\rzero,0) node[below]{$r_0$} -- 
(\rzero, {1/((\rzero * \mlogrzero * \rzero * \mlogrzero ))});

\draw[dashed, thin] (-\rzero,0) -- 
(-\rzero, {1/((\rzero * \mlogrzero * \rzero * \mlogrzero ))});

\draw[dashed, thin] ({1/\ee},0) node[below]{$e^{-1}$} -- 
({1/\ee}, {\ee * \ee});

\draw[dashed, thin] ({- 1/\ee},0)  -- 
({-1/\ee}, {\ee * \ee});
\end{tikzpicture}
\caption{A cigar (solid) tangent to the hyperbolic cusp (dashed)}
\label{fig:cigar}
\end{figure}

\begin{figure}
\centering
\begin{tikzpicture}[xscale=15,yscale=0.05]

\newcommand*{\rlimit}{0.4}
\newcommand*{\axisheight}{90}
\newcommand*{\rzero}{0.2}
\newcommand*{\mlogrzero}{-ln(\rzero)}
\newcommand*{\epsi}{1/(\mlogrzero*(\mlogrzero-1))}
\newcommand*{\delt}{\rzero*\rzero/(\mlogrzero-1)}
\newcommand*{\ee}{2.71828}

\draw [<->] (0,\axisheight) -- (0,0) -- (0.43,0) node[below]{$r$} ;
\draw [->] (0,0) -- (-0.43,0);

%
%
\draw[dashed, domain=0.03:\rlimit, samples=\sampno] 
plot (\x, {
1/(\x * ln(\x))^2
}); 

\draw[dashed, domain=-\rlimit:-0.03, samples=\sampno] 
plot (\x, {
1/(\x * ln(-\x))^2
});

%
%
\draw[domain=-\rlimit:\rlimit, samples=2*\sampno] 
plot (\x, {
\epsi/(\delt+(\x * \x))
});
%
%
\renewcommand*{\rzero}{0.15}
\draw[domain=-\rlimit:\rlimit, samples=2*\sampno] 
plot (\x, {
\epsi/(\delt+(\x * \x))
});
%
%
\renewcommand*{\rzero}{0.10}
\draw[domain=-\rlimit:\rlimit, samples=2*\sampno] 
plot (\x, {
\epsi/(\delt+(\x * \x))
});
%
%
\renewcommand*{\rzero}{0.08}
\draw[domain=-\rlimit:\rlimit, samples=2*\sampno] 
plot (\x, {
\epsi/(\delt+(\x * \x))
});
%
%
\renewcommand*{\rzero}{0.07}
\draw[domain=-\rlimit:\rlimit, samples=2*\sampno] 
plot (\x, {
\epsi/(\delt+(\x * \x))
});
%
%

\draw[dashed, thin] ({1/\ee},0) node[below]{$e^{-1}$} -- 
({1/\ee}, {\ee * \ee});

\draw[dashed, thin] ({- 1/\ee},0)  -- 
({-1/\ee}, {\ee * \ee});
\end{tikzpicture}
\caption{The cusp (dashed) as the envelope of cigars (solid)}
\label{fig:env}
\end{figure}
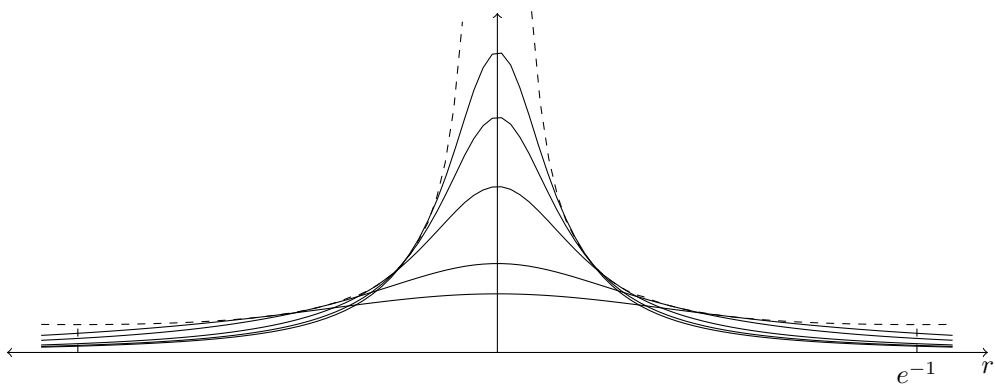

For $r_0<1/e$, we want to solve for $\varepsilon$ and $\delta$ so that the conformal factors of the cigar and the cusp are tangent at $r=r_0$. This is equivalent to the following two equations, asserting that the metrics agree
\begin{equation*}
\frac{\varepsilon}{\delta+r_0^2}= \frac{1}{r_0^2(\log r_0)^2}  
\end{equation*}
as do their first derivatives
\begin{equation*}
\varepsilon\frac{2r_0}{(\delta+r_0^2)^2}=  2\frac{(-\log r_0-1)}{r_0^3(-\log r_0)^3}.
\end{equation*}
Dividing the first equation by the second gives
\begin{equation}\label{eqn:deltar0}
	\delta+r_0^2=\frac{r_0^2(-\log r_0)}{(-\log r_0-1)},
\end{equation}
from which we get
\begin{equation}\label{eqn:delta}
	\delta(r_0)=\frac{r_0^2(-\log r_0)}{(-\log r_0-1)}-r_0^2= \frac{r_0^2}{-\log r_0-1}
\end{equation}
and
\begin{equation}\label{eqn:varepsilon}
	\varepsilon(r_0)=(\delta+r_0^2)\frac{1}{r_0^2(\log r_0)^2}=\frac{1}{(-\log r_0-1)(-\log r_0)}.
\end{equation}
The next lemma implies that the cigar metric corresponding to $\varepsilon(r_0)$ and $\delta(r_0)$, i.e.
\begin{equation}
	\tilde{u}_{r_0}(\bx):=\frac{\varepsilon(r_0)}{\delta(r_0)+ r^2},
	\label{eqn:ur0}
\end{equation}
where $\bx=(x,y)$ so that $r=\abs{\bx}$,
lies below the cusp and touches it only on $\set{r=r_0}$. The proof, being very elementary computations, is moved to the appendix. 
\begin{lemma}\label{lem:touch}
For $0<r_0<e^{-1}$, and 
with $\delta(r_0)$ and $\varepsilon(r_0)$ given as above, we have
%
%
\begin{equation*}
	\frac{\varepsilon}{\delta+r^2}\leq \frac{1}{r^2(\log r)^2}
\end{equation*}
for all $0<r<1$ with equality only at $r=r_0$.
\end{lemma}
With the help of $\tilde{u}_{r_0}$, we define for $r_0\in (0,e^{-1})$,
a function $u_{r_0}:B\to\R$ by
\begin{equation*}
	u_{r_0}(\bx) = \left\{
		\begin{array}[]{ll}
			\tilde{u}_{r_0}(\bx) & \quad 0\leq r\leq r_0 \\
			h(\bx) & \quad r_0<r<1.
		\end{array}
		\right.
\end{equation*}
The graph of $u_{r_0}$ is shown in Figure \ref{fig:cigar} as the solid line (the cigar metric) for $r<r_0$ and the dashed line (the hyperbolic cusp metric) for $r\geq r_0$. By our choice of $\varepsilon$ and $\delta$, $u_{r_0}$ is a $C^1$ function. Moreover, it is obvious from the graph or a simple argument that 
\begin{equation}
	u_{r_0}(r)\geq u_{r_0}(e^{-1})=e^{2}
	\label{eqn:minimal}
\end{equation}
for all $r\in [0,1)$.
Now, for any $r_n\downto 0$, the sequence $u_n:=u_{r_n}$ 
is an approximation of $h$. Besides the obvious fact that for any $r\in (0,1)$, $u_n=h$ on any $B\setminus B_r$ for sufficiently large $n$, we note that
\begin{equation}\label{eqn:approx}
	\lim_{n\to \infty} \int_{B} \abs{u_n-h} d\bx =0
\end{equation}
and that:
\begin{lemma}
The functions $u_{r_0}:B\to\R$ are decreasing in $r_0$. In particular, 
	$\{u_n\}$ is an increasing sequence of functions $B\to\R$.
	\label{lem:increase}
\end{lemma}
The simple proof of this lemma is postponed to the appendix.

\subsection{Existence of $g_{cc}$}\label{sec:existence}
In this section, we confirm the existence of Theorem \ref{exist_unique}, taking a shortcut compared with \cite{revcusp} by exploiting the new estimate in \cite{TY}. The construction will also imply the upper bound of Theorem \ref{secondthm} automatically.

The local existence of a Ricci flow starting from $u_{r_0}(dx^2+dy^2)$ is known by a result of Shi \cite{Shi}. The global existence of a solution $u_{r_0}(t)(dx^2+dy^2)$, for $t\in [0,\infty)$,
was proved by Giesen and the first author in \cite{GT2}. 
More generally, the solutions from \cite{GT2} for general initial data, together with their property of being \emph{maximally stretched} (see \cite{GT2}) and the uniqueness of \cite{ICRF_UNIQ}, gives us the following comparison principle (as explained in 
\cite[\S 1]{ICRF_UNIQ}).
\begin{lemma}
\label{max_stretch_lem}
Suppose $u(t)$ and $\tilde u(t)$ are the conformal factors of two smooth Ricci flows on some underlying Riemann surface for $t\in [0,T]$ and $u(0)\geq \tilde u(0)$. If the Ricci flow of $u(t)$ is complete, then $u(t)\geq \tilde u(t)$ for all $t\in [0,T]$.
\end{lemma}
\brmk
\label{max_stretch_conseq}
A first application of this lemma, combined with Lemma \ref{lem:increase}, is that if 
$0<r_0\leq\tilde r_0<e^{-1}$ then
$u_{\tilde r_0}(t)\leq u_{r_0}(t)$.
\ermk
We wish to consider flows $u_{r_0}(t)$ for $r_0$ equal to each $r_n\downto 0$ considered above, and we denote this increasing sequence by $u_n(t):=u_{r_n}(t)$. 
We claim that 
	
	\begin{enumerate}[(a)]
		\item a subsequence (still denoted by $u_n$) converges to another 
		solution $u_{cc}$ smoothly locally on $B\times (0,\infty)$;
		\item as $t\downto 0$, $u_{cc}(t)$ converges to $h$ smoothly locally on $B\setminus \set{0}$;
		\item for any $\ck\subset\subset B$,
	\begin{equation}\label{eqn:L1initial}
		\lim_{t\downto 0} \int_{\ck} \abs{u_{cc}(t)-h} d\bx =0.
	\end{equation}
	\end{enumerate}
	The metric $g_{cc}(t):=u_{cc}(t)(dx^2+dy^2)$ would then be the contracting cusp Ricci flow solution whose existence is asserted by Theorem \ref{exist_unique}. 
Note that the limit $u_{cc}(t)$ would be automatically complete
because the sequence $u_n(t)$ is increasing, and so
$u_n(t)\leq u_{cc}(t)$.
	The rest of this section is devoted to the proof of this claim
(a) to (c).

The key to proving (a) is obtaining upper and lower bounds on the conformal factors $u_n$. We will then be able to apply parabolic regularity theory. The following simple lemma gives the lower bound, and upper bounds away from the origin.
\begin{lemma}
\label{upper_and_lower_un}
For all $t\geq 0$, we have
\begin{equation}
\label{un_lower}
u_n(\cdot,t)\geq e^2\qquad\text{ in } B,
\end{equation}
and 
\begin{equation}
\label{unh}
u_n(\cdot,t)\leq (1+2t)h\qquad\text{ in } B\backslash\{0\}.
\end{equation}
\end{lemma}
\begin{proof}
The lower bound \eqref{un_lower} follows immediately from 
\eqref{eqn:minimal} together with Lemma \ref{max_stretch_lem}, because $u_n(\cdot,t)$ and $e^2$ are the conformal factors of two Ricci flows on $B$,
and the former is complete.
Meanwhile, the upper bound also follows from Lemma \ref{max_stretch_lem}, but this time because $(1+2t)h$ is the conformal factor of a complete Ricci flow on $B\setminus \set{0}$, and $u_n(\cdot,0)\leq h$.
\end{proof}
The more subtle issue is to obtain upper bounds near the origin, but this will follow from 
an application of the $L^1-L^\infty$ estimate proved in \cite{TY}.
	\begin{lemma}\label{lem:smoothing}
	Suppose $u:B\times [0,T)\to (0,\infty)$ is a smooth solution to the equation \eqref{eqn:rfu} with
		\begin{equation*}
			u_0:=u(\cdot,0)\leq h
		\end{equation*}
		on $B\setminus \set{0}$. 
Then 
		\begin{equation*}
			\log u\leq \frac{2}{t} +C
		\end{equation*}
		on $B_{1/2}$ for $t\in (0,\min\{1,T\})$, where $C$ is universal.
\end{lemma}

\begin{rmk}
In \cite{revcusp} it was shown, for example, that a surface with a hyperbolic cusp can be evolved under Ricci flow by allowing the cusp to collapse.  The key \emph{a priori} estimate there, in the language of this paper,  was essentially a weaker version of the above lemma in the sense that the upper bound for $\log u$ was $C/t$ with $C$ some uncontrolled universal constant $C$.
\end{rmk}

We will give a proof of Lemma \ref{lem:smoothing} based on the following new estimate from \cite{TY}, describing the evolution compared with the Poincar\'e metric on $B$ with conformal factor
	\begin{equation*}
		\tilde{h}(\bx)=\left( \frac{2}{1- r^2} \right)^2.
	\end{equation*}

\begin{thm}[{Special case of \cite[Theorem 1.3]{TY}}]
\label{L1Linfthm}
Suppose $u:B\times [0,T)\to (0,\infty)$ is a smooth solution to the equation
$\pt u=\lap\log u$, with initial data $u_0:=u(0)$
on the unit ball $B\subset\R^2$, and suppose that $(u_0-\al \tilde h)_+\in L^1(B)$
for some $\al\geq 1$. 
Then for all $\de>0$ (however small) and any time $t\in [0,\min\{1,T\})$ satisfying 
$$t\geq \frac{\|(u_0-\al \tilde h)_+\|_{L^1(B)}}{4\pi}(1+\de),
\quad\text{ we have }\quad
u(t)\leq C\al{\tilde h} \quad\text{ throughout }B,$$
where $C<\infty$ depends only on $\de$.
\end{thm}
The special case of our result that we have given here focuses on the relevant case for us that 
$\|(u_0-\al \tilde h)_+\|_{L^1(B)}$ is small, and indeed the theorem is vacuous unless 
$\|(u_0-\al \tilde h)_+\|_{L^1(B)}< 4\pi/(1+\de)$. See \cite{TY} for a more general statement.

\begin{proof}[{Proof of Lemma \ref{lem:smoothing}}]
For $t\in (0,\min\{1,T\})$, set $r=e^{-1/t}\in (0,e^{-1}]$ and
$$\al=\left(\frac{1}{r\log r}\right)^2\leq \frac{1}{r^2}=e^{2/t}>1.$$
Coarse estimation and precise computation (see Figure \ref{fig3})
gives
$$\|(u_0-\al \tilde h)_+\|_{L^1(B)}\leq \|h-\al \|_{L^1(B_{r})}
=\frac{2\pi}{-\log r}-\frac{\pi}{(\log r)^2}
\leq \frac{2\pi}{-\log r}=2\pi t,$$
and so Theorem \ref{L1Linfthm} applies with $\de=1$ to give
$$u(t)\leq C\al \tilde h = Ce^{2/t}\tilde h,$$
and hence
$$\log u(t)\leq C+\log \tilde h +\frac{2}{t},$$
throughout $B$, for universal constants $C$. 
\end{proof}

\begin{figure}
\centering
\begin{tikzpicture}[xscale=5.7,yscale=0.07]

\newcommand*{\rlimit}{0.89}
\newcommand*{\tval}{0.6}
\newcommand*{\rval}{(exp(-1.0/\tval))}
\newcommand*{\rvalinv}{(exp(1.0/\tval))}
\newcommand*{\alphaval}{1/((\rval*(ln(\rval)))^2)}
\newcommand*{\axisheight}{90}

\draw [<->] (0,\axisheight) -- (0,0) -- (1.1,0);
\draw [->] (0,0) -- (-1.1,0);

\draw[ultra thick, domain=0.0322:\rlimit, samples=6*\sampno] 
plot (\x, {
1/(\x * ln(\x))^2
}) node[above]{$h$};

\draw[ultra thick, domain=-\rlimit:-0.03, samples=6*\sampno] 
plot (\x, {
1/(\x * ln(-\x))^2
});

\draw[ultra thick, domain=-0.6:0.6, samples=6*\sampno] 
plot (\x, {
(4*\alphaval)/((1-\x * \x)*(1-\x * \x))
}) node[above]{$\al \tilde{h}$};

\draw [dashed, thin] ({\rval},0) node[below]{$r$} -- ({\rval},{(\rvalinv/ln(\rval))^2});
\draw [dashed, thin] ({-\rval},0) -- ({-\rval},{(\rvalinv/ln(\rval))^2});

\draw [dashed, thin] ({exp(-1)},0) node[below]{$e^{-1}$} -- ({exp(-1)},
{(exp(1))^2});

\draw [dashed, thin] (0,{\alphaval}) node[below left]{$\al$} -- ({\rval},{\alphaval});
\draw [dashed, thin] (0,{\alphaval}) -- (-{\rval},{\alphaval});

\draw [dashed, thin] (-1,{4*\alphaval}) node[left]{$4\al$} -- (1,{4*\alphaval});

\draw [dashed, thin] (-1,0) node[below]{$-1$} -- (-1,\axisheight);
\draw [dashed, thin] (1,0) node[below]{$1$} -- (1,\axisheight);

\draw[thick, domain=-1:1, samples=6*\sampno] 
plot (\x, {
2+5*\x*\x+2*sin(deg(10*\x*\x))+0.04/(0.0006+(\x * \x))
}) node[right]{$u_0$};

\end{tikzpicture}
\caption{Hyperbolic metrics}
\label{fig3}
\end{figure}
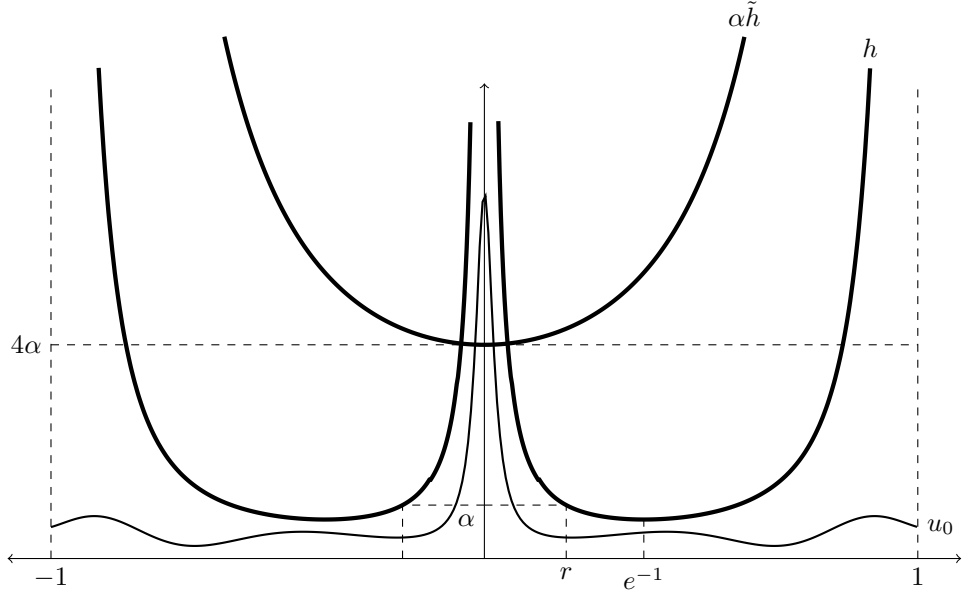

Note that the upper bound deteriorates as $t\downto 0$. However, we can apply it to the solutions $u_n$ to give
\beq
\label{un_upper}
\log u_n(t)\leq \frac{2}{t} +C, 
\eeq
on $B_{1/2}$, for $t\in (0,1)$. 
From time $t=1$ onwards, we can then compare $u_n$ against the Ricci flow $(M+2t)\hat h$, where $\hat h$ is the complete hyperbolic (Poincar\'e) metric on $B_{1/2}$ and $M$ is chosen sufficiently large so that this flow majorises $u_n$ at time $t=1$.
Together with Lemma \ref{upper_and_lower_un}, we then conclude that
for all $\ck\subset\subset B$ and $\ep>0$, we have  upper and lower bounds for $u_n$ on $\ck\times [\ep,1/\ep]$ that are uniform in $n$.
Standard parabolic regularity theory then 
generates the higher-order estimates from which (a) follows.
More precisely, we have
a choice of applying parabolic theory to the equation \eqref{eqn:rfu} for $u_n$ or \eqref{eqn:rfv} for $v_n:=\half\log u_n$. In the former case, one can start with \cite[Theorem III.10.1]{Lady}, and in the latter case with \cite[Theorem V.1.1]{Lady}, followed by Schauder estimates.

To prove (b), it suffices to establish upper and lower bounds for the conformal factors $u_n$ on an arbitrary $\ck\subset\subset B\backslash\{0\}$ that are uniform all the way back to time $t=0$, because for sufficiently large $n$, we have $u_n(0)=h$ on $\ck$, and parabolic regularity theory will then give uniform higher-order estimates for $u_n$ down to $t=0$, which then imply that the convergence $u_n\to u_{cc}$
is smooth local convergence all the way down to $t=0$.
This time, the bounds for the conformal factor are immediate 
from Lemma \ref{upper_and_lower_un},
which gives
$$\sup_{\ck\times [0,1]}u_n \leq \sup_{\ck}3h<\infty.$$ 
%

Finally we turn to (c).
Because the flows $u_n(t)$ are increasing with $n$, we know that $u_n\leq u_{cc}$. 
On the other hand, by passing to the limit $n\to\infty$ in \eqref{unh}, 
we have $u_{cc}(t)\leq (1+2t)h$. Using these inequalities, we can estimate 
\begin{equation}
\label{triangleineq}
\begin{aligned}
|u_{cc}(t)-h|&\leq |u_{cc}(t)-(1+2t)h|+2th = \left((1+2t)h-u_{cc}(t)\right)+2th\\
&\leq \left(h-u_n(t)\right)+4th\\ 
&\leq |h-u_n(0)|+|u_n(0)-u_n(t)|+4th.
\end{aligned}
\end{equation}
Now for any $\ck\subset\subset B$ and $\ep>0$, by \eqref{eqn:approx} we can fix $n$ large so that 
$$\int_\ck|h-u_n(0)|\leq \ep/3.$$
For this fixed $n$, by smoothness of $u_n$, there exists $t_0>0$ small enough so that
$$\int_\ck |u_n(0)-u_n(t)|\leq \ep/3\quad\text{ for all }t\in [0,t_0],$$
and 
$$4t_0\int_\ck h\leq \ep/3.$$
Therefore, by integrating \eqref{triangleineq}, we have
$$\int_\ck |u_{cc}(t)-h|\leq \ep,$$
for $t\in [0,t_0]$,
which is enough to conclude (c).
\begin{rmk}\label{rmk:lower}
	In the limit $n\to\infty$, the uniform lower bound  
$u_n(t)\geq e^2$ from \eqref{un_lower} implies a lower bound of $u_{cc}$, i.e. $u_{cc}(\bx,t)\geq e^2$ for all $\bx\in B$ and $t\geq 0$.
\end{rmk}

\begin{rmk}\label{rmk:upper}
	We can also
take a limit $n\to\infty$ in the uniform upper bound \eqref{un_upper}
to deduce the upper bound of Theorem \ref{secondthm}.
\end{rmk}

\brmk
\label{new_sandwich_rmk}
Additional  upper and lower bounds for $u_{cc}(t)$ hold thanks to 
Lemma \ref{max_stretch_lem}. First, observe that
$\tilde h\leq h$ either by computing or by applying
the comparison principle.
Hence by construction of $u_n$ we have $\tilde h\leq u_n(0)$ throughout $B$.
Therefore,
we have $(1+2t)\tilde h\leq u_{n}(t)$, and in the limit $n\to\infty$, this gives $(1+2t)\tilde h\leq u_{cc}(t)$. Meanwhile, by comparing $u_{cc}(t)$  and $(1+2t)h$ on $B\backslash\{0\}$, or simply passing to the limit $n\to\infty$ in \eqref{unh} as in the proof of (c), we obtain
$u_{cc}(t)\leq (1+2t)h$.
\ermk

\subsection{Uniqueness}
\label{unique_sect}

In this section, we 
prove the uniqueness assertion of Theorem \ref{exist_unique}.
One implication of this uniqueness result is that we have some freedom in choosing the approximation sequence $u_n$ used in the construction of $g_{cc}$.
The proof will use the following result that follows easily from 
Lemma 5.1 of \cite{ICRF_UNIQ}. 
\begin{lemma}
\label{even_stronger_again_lemma}
Suppose $1/2<r_0<r_0^{1/3}<R<1$, and $\gamma\in (0,\half)$.
Suppose $g_1(t)=u_1(t)(dx^2+dy^2)$ is any complete Ricci flow on $B$, 
and $g_2(t)=u_2(t)(dx^2+dy^2)$ is any Ricci flow 
on $\overline {B_R}$, both for $t\in (0,T]$. 
Then we have for all $t\in (0,T]$ that
\begin{equation}
\begin{aligned}
\left[\int_{B_{r_0}}(u_2(t)-u_1(t))_+d\bx\right]^{\frac{1}{1+\ga}}
&\leq
\liminf_{s\downto 0}
\left[\int_{B_{R}}(u_2(s)-u_1(s))_+d\bx\right]^{\frac{1}{1+\ga}}\\
&\quad+C(\ga)
\left[\frac{t}{
(-\log r_0)\left[\log(-\log r_0)-\log(-\log R)\right]^\ga}
\right]^{\frac{1}{1+\ga}}.
\end{aligned}
\end{equation}
\end{lemma}
\begin{cor}
\label{uniq_cor}
Suppose $g(t)=u(t)(dx^2+dy^2)$ and $\tilde g(t)=\tilde u(t)(dx^2+dy^2)$
are complete Ricci flows on $B$ for $t\in (0,T]$ with 
$u(t)-\tilde u(t)\to 0$ in $L^1_{loc}$ as $t\downto 0$.
Then $g(t)=\tilde g(t)$ for all $t\in (0,T]$.
\end{cor}
Clearly the uniqueness assertion of Theorem \ref{exist_unique} is just one instance of this general corollary.
\begin{proof}[Proof of Corollary \ref{uniq_cor}]
First, we apply Lemma \ref{even_stronger_again_lemma} with 
$g_1(t)=\tilde g(t)$ and $g_2(t)=g(t)$
and $\ga=1/4$ to establish that
\begin{equation*}
\int_{B_{r_0}}(u(t)-\tilde u(t))_+d\bx
\leq C
\frac{t}{
(-\log r_0)\left[\log(-\log r_0)-\log(-\log R)\right]^{1/4}}
\end{equation*}
for all $t\in (0,T]$ and $1/2<r_0<r_0^{1/3}<R<1$.
By taking the limits $R\upto 1$ and then $r_0\upto 1$, we find that
$u(t)\leq \tilde u(t)$ throughout $B$, for all $t\in (0,T]$.
By repeating the argument with $g(t)$ and $\tilde g(t)$ interchanged, we conclude that $g(t)=\tilde g(t)$.
\end{proof}

\subsection{Proof of Theorem \ref{secondthm}}

We have already proved the upper bound of Theorem \ref{secondthm} -- see Remark \ref{rmk:upper}.
For the lower bound, we use the approximation in Section \ref{sec:approx} again. The advantage is that the approximations $u_n:B\to (0,\infty)$ are envelopes of families of cigar metrics whose evolution under the Ricci flow can be written down explicitly. More precisely, the complete Ricci flow starting from $\tilde{u}_{r_0}$ is given by
\begin{equation}\label{eqn:cigarsolution}
	\tilde{u}_{r_0}(\bx,t)=\frac{\varepsilon(r_0)}{\delta(r_0) e^{4t/\varepsilon(r_0)}+ r^2}.
\end{equation}
By Remark \ref{max_stretch_conseq},
we have for $t>0$,
\begin{equation*}
	u_{cc}(0,t)= \lim_{n\to \infty} u_n(0,t) = \sup_{r_0\in (0,e^{-1})} u_{r_0}(0,t).
\end{equation*}
By the definition of $u_{r_0}$, we know that $\tilde{u}_{r_0}(\bx)\leq u_{r_0}(\bx)$ 
and hence, by Lemma \ref{max_stretch_lem},  $\tilde{u}_{r_0}(\bx,t)\leq u_{r_0}(\bx,t)$, which implies
\begin{equation*}
	u_{cc}(0,t)\geq  \max_{r_0\in (0,e^{-1})} \tilde{u}_{r_0}(0,t).
\end{equation*}
Thanks to the explicit formula \eqref{eqn:cigarsolution}, and \eqref{eqn:delta} and \eqref{eqn:varepsilon}, we can compute
\begin{equation}
\label{logexpression}
	-\log \tilde{u}_{r_0}(0,t) = 2 \log r_0 + \log (-\log r_0) + 4t (-\log r_0)(-\log r_0-1).
\end{equation}
We would like to maximise this over $r_0\in (0,e^{-1})$, so we compute
$$\frac{\partial}{\partial r_0}(-\log \tilde{u}_{r_0}(0,t))
=\frac{(-2\log r_0-1)}{r_0}\left[\frac{1}{-\log r_0}-4t\right],$$
and deduce that for $t\in (0,1/4)$ we should set $r_0=e^{-1/(4t)}$ 
in \eqref{logexpression} to obtain the optimal estimate
\begin{equation*}
	u_{cc}(0,t)\geq 4t e^{\frac{1}{4t}+1}.
\end{equation*}
The lower bound in Theorem \ref{secondthm} then follows from the fact that $v_{cc}(t)=\frac{1}{2}\log u_{cc}(t)$.


\section{A Li-Yau differential Harnack estimate and the curvature upper bound}

In this section, we prove a new Li-Yau Harnack estimate for the Ricci flow, Theorem \ref{thm:liyau}, which allows us to prove the curvature upper bound in Theorem \ref{mainthm}. 
Compared with other known Li-Yau Harnack inequalities \cite{CaoHamilton}, 
our argument in Section \ref{sec:liyau} will exploit an additional geometric property of the initial metric (the hyperbolic cusp metric $h(dx^2+dy^2)$), which is given by the inequality

\begin{equation}\label{eqn:special}
	h^{-1} \abs{\nabla \log \left( \frac{1}{2} \log h\right)}^2\leq C
\end{equation}
for some universal constant $C>0$, as formally asserted in Lemma \ref{lem:Fn}. We note that the computation and argument in Section \ref{sec:liyau} still works without this extra information, to give an upper bound of curvature, but it is not the sharp bound we give here.

In fact, we will not apply the Li-Yau-type estimate starting at $t=0$, but at some later time.
This gives rise to a technical issue that we want \eqref{eqn:special} to remain true for $u_{cc}(t)$ for a period of time instead of just for $t=0$. This is proved in Lemma \ref{lem:special} by using the approximations $u_n$ defined in Section \ref{sec:approx}, albeit 
only on $B_\half$.

\subsection{Gradient bounds for the approximations $u_n$}
The aim of this section is to prove Lemma \ref{lem:special}, which is a gradient bound property of $u_{cc}$ alone. However, the proof relies on the special choice of approximations $u_n$ in Section \ref{sec:approx}. In some sense, $u_n$ smooths out the singularity of $h$ at the origin while keeping \eqref{eqn:special} valid. 

In what follows, we use the subscript $u$ in $\abs{\nabla f}_u$, $\triangle_u$ and $\langle \cdot,\cdot \rangle_u$ to indicate that the gradient norm, the Laplacian and the inner product are taken with respect to the metric $u (dx^2+dy^2)$, whereas no subscript means that the gradient norm, the Laplacian and the inner product are those of the flat metric $dx^2+dy^2$ on $B$. 


\begin{lemma}\label{lem:Fn}
There exists a universal constant $C>0$ such that \eqref{eqn:special} holds, and also so that 
\begin{equation}\label{Fn_ineq}
	F_n:= \abs{\nabla f_n}_{u_n}^2 \leq C
\end{equation}
throughout $B$, where $f_n:= \log v_n$, $v_n:=\frac{1}{2}\log u_n\geq 1$ by \eqref{eqn:minimal}, and $u_n:B\to (0,\infty)$ is defined as in Section \ref{sec:approx}.
\end{lemma}
In other words, not only does \eqref{eqn:special} hold for $h$, it also holds for the approximations $u_n$.
Since we have explicit formulae for $h$ and $u_n$, the proof is a straightforward computation and is moved to the appendix.

The estimate \eqref{Fn_ineq} holds not only for $u_n$, but also for its evolution $u_n(t)$, with $t\in (0,1]$, and the corresponding $F_n=F_n(t)$. 
(Note that the evolving $F_n$ makes sense because $v_n(t)=\half\log u_n(t)\geq 1$ by 
\eqref{un_lower}.)
This is proved using the maximum principle for domains with boundary. Therefore we must argue first that this $F_n$ is bounded on the boundary of $B_{1/2}$ for $t\in [0,1]$. This is the aim of the next 
lemma.

\begin{lemma}\label{lem:Fnboundary}
	There exists a universal constant $C>0$ such that 
	\begin{equation*}
		\sup_{\partial B_{1/2}\times [0,1]} F_n \leq C
	\end{equation*}
	where 
	$F_n:= \abs{\nabla f_n}_{u_n}^2$,  	
	$f_n=\log (\frac{1}{2}\log u_n)$ and 
	$u_n=u_n(t)$ is the Ricci flow evolution of $u_n:B\to (0,\infty)$.
\end{lemma}

\begin{proof}
Lemma  \ref{upper_and_lower_un} tells us that 
there exists a universal constant $C>0$ such that 
\begin{equation}\label{eqn:uc0}
e^2 \leq u_n(\by,t)\leq C 	
\end{equation}
for all $t\in [0,1]$ and $\by\in B_{5/8}\setminus B_{3/8}$.
By the definition of $f_n$ and \eqref{eqn:uc0}, we know it suffices to prove
	\begin{equation*}
		\sup_{\partial B_{1/2}\times [0,1]} \abs{\nabla u_n}^2
		\leq C.
	\end{equation*}
	Indeed, we have $\norm{u_n}_{C^{2,\alpha}(B_{17/32}\setminus B_{15/32}\times [0,1])}$ for some $\alpha\in (0,1)$ bounded by a universal number as well. To see this, we first note that $u_n(0)$ is nothing but $h$ on the annulus 
$B_{5/8}\setminus B_{3/8}$ for large $n$, hence we have uniform bounds of all derivatives at $t=0$. 
We will then be able to appeal to parabolic regularity theory as before:
The estimate \eqref{eqn:uc0}
implies that $\partial_t u_n = \triangle \log u_n$ is uniformly parabolic on the annulus 
$B_{5/8}\setminus B_{3/8}$
and hence $\norm{u_n}_{C^\alpha(B_{9/16}\setminus B_{7/16}\times [0,1])}$ is uniformly bounded.  The rest follows from the linear Schauder theory. 
\end{proof}

Finally, we can prove the property mentioned at the beginning of this section.
\begin{lemma}
	There exists a universal constant $C_1>0$ such that
	\begin{equation*}
		\abs{\nabla \log v_{cc}}^2_{u_{cc}}\leq C_1
	\end{equation*}
	on $B_{1/2}\times (0,1]$.	
	\label{lem:special}
\end{lemma}
\begin{proof}
	Since $u_n$ converges to $u_{cc}$ smoothly, locally on $B\times (0,1]$, it remains to prove $F_n\leq C_1$ on $B_{1/2}\times [0,1]$ where $F_n=\abs{\nabla f_n}^2_{u_n}$ is defined in Lemma \ref{lem:Fn}. 	
Direct computation shows that $F_n$ satisfies 
\begin{equation*}
	(\partial_t - \triangle_{u_{n}}) F_n = -2 \abs{\Hess f_n}^2_{u_{n}} - 2\langle\nabla F_n,\nabla f_n\rangle_{u_n}.
\end{equation*}
By Lemma \ref{lem:Fn} and Lemma \ref{lem:Fnboundary}, we can apply the classical maximum principle for domains with boundary  to $F_n$ to see that
\begin{equation}\label{eqn:FC1}
	\sup_{B_{1/2}\times [0,1]} F_n \leq C_1.
\end{equation}
\end{proof}

\subsection{A Li-Yau differential Harnack inequality}\label{sec:liyau}

Next, we present a Li-Yau type Harnack inequality for solutions of \eqref{eqn:rfv}. Note that $v_{cc}$ satisfies a `linear' heat equation with background metric evolving as a Ricci flow. Such Harnack inequalities are known in various cases, for example, Lemma 2.1 of \cite{CaoHamilton}. Here $v_{cc}$ is not only the solution to the heat equation, but also the conformal factor of the Ricci flow. This dual role of $v_{cc}$ helps us to remove the curvature assumption in \cite{CaoHamilton}. 


\begin{thm}[Li-Yau type estimate]
	\label{thm:liyau}
	Let $v(t)$ be a solution to \eqref{eqn:rfv} on $B_{1/2}\times [0,t_0]$ with $t_0<1/2$ and $u(t)=e^{2v(t)}$. If 
	\begin{equation}\label{eqn:assumption}
		v(t)\geq 1\quad \mbox{and} \quad \abs{\nabla \log v}_u^2\leq C_2
	\end{equation}
	on $B_{1/2}\times [0,t_0]$ for some $C_2>0$ and
	\begin{equation}\label{eqn:time}
		t_0\leq\frac{1}{8 \max_{B_{1/2}\times [0,t_0]} v},
	\end{equation}
	then there exists $C_3>0$ depending only on $C_2$ such that the Gauss curvature is controlled by
	\begin{equation*}
		K\leq \frac{C_3 v}{t}
	\end{equation*}
	on $B_{1/4}$ for $0<t\leq t_0$.
\end{thm}

\begin{rmk}
	The result of this theorem is sharp in the sense that when combined with Theorem \ref{secondthm}, it yields Theorem \ref{mainthm}, which is sharp as far as the order of $t$ is concerned. 
\end{rmk}

The proof of this theorem starts with a computation similar to that of Li and Yau in \cite{LiYauActa}. However, we use a different $F$, taking advantage of the second part of \eqref{eqn:assumption}. Writing $f$ for $\log v$ and keeping in mind that
\begin{equation}
\label{dtf}
	\partial_t f = \triangle_u f + \abs{\nabla f}_u^2\quad \mbox{and}\quad \partial_t f = \frac{\partial_t v}{v} =\frac{-K}{v},
\end{equation}
by \eqref{eqn:rfv}, we make the equivalent definitions
\begin{eqnarray}
	\label{F1} F&=&  \abs{\nabla f}^2_u -  t \partial_t f \\
	\label{F2} &=&   \abs{\nabla f}^2_u +\frac{Kt}{v} \\
	\label{F3} &=& -\triangle_u f + (1- t) \partial_t f \\
	\label{anotherF} &=& -\triangle_u f - (1- t) \frac{K}{v} \\
	\label{F4} &=& (1-t) \abs{\nabla f}_u^2 - t \triangle_u f.
\end{eqnarray}
By \eqref{F2}, 
Theorem \ref{thm:liyau} reduces to the claim that $F$ is bounded from above on $B_{1/4}\times (0,t_0]$. 
At $t=0$, \eqref{eqn:assumption} implies that $F\leq C_2$ throughout $B_{1/2}$, so we would like to establish whether an upper bound persists.
As in \cite{LiYauActa}, we first derive the evolution equation for $F$.
Direct computation shows
\begin{equation*}
	\triangle_u F= 2\abs{\Hess_u f}_u^2 +  2 \langle \nabla \triangle_u f, \nabla f\rangle_u + 2K \abs{\nabla f}_u^2 -  t \triangle_u \partial_t f 
\end{equation*}
and
\begin{equation*}
\begin{aligned}
	\partial_t F &=  - \abs{\nabla f}_u^2 -  \triangle_u f \\
	& \quad +  (1- t) \left( 2K\abs{\nabla f}_u^2 + 2\langle \nabla f, \nabla \partial_t f\rangle_u\right)  -  t \partial_t \triangle_u f,
\end{aligned}
\end{equation*}
which combine to give
\begin{equation*}
\begin{aligned}
	\partial_t F -\triangle_u F &= -2 \abs{\Hess_u f}_u^2 - 2 \langle\nabla\triangle_u f,\nabla f\rangle_u + 2(1- t) \langle\nabla f,\nabla \partial_t f\rangle_u\\
	&\quad -(1+2 K t) \abs{\nabla f}_u^2  \\
	&\quad + t \left( \triangle_u \partial_t f - \partial_t \triangle_u f \right) - \triangle_u f.
\end{aligned}
\end{equation*}
We simplify the above by using \eqref{F3} and computing $\triangle_u \partial_t f - \partial_t \triangle_u f = 2 \triangle_u f \partial_t v =-2K \triangle_u f$, which in turn used \eqref{eqn:rfv}, to get
\begin{equation*}
	\partial_t F -\triangle_u F = -2 \abs{\Hess_u f}_u^2 + 2 \langle \nabla F, \nabla f\rangle_u -(1+2 K t) \left(\abs{\nabla f}_u^2   + \triangle_u f\right). 
\end{equation*}
By \eqref{dtf} we then have 
\begin{equation*}
	\partial_t F -\triangle_u F = -2\abs{\Hess_u f}_u^2 + 2\langle\nabla f,\nabla F\rangle_u + (1+2Kt) \frac{K}v. 
\end{equation*}
Meanwhile, inserting \eqref{anotherF} in the Schwartz inequality
\begin{equation*}
\begin{aligned}
	2\abs{\Hess_u f}_u^2 &\geq (\triangle_u f)^2 = 
	\left( F + (1- t)\frac{K}{v} \right)^2 \\
	&= F^2+\frac{2(1- t)FK}{v} + \frac{(1- t)^2 K^2}{v^2},
\end{aligned}
\end{equation*}
we obtain
\begin{equation}\label{eqn:Finequality}
	\partial_t F-\triangle_u F \leq - F^2-\frac{2(1- t)FK}{v} - \frac{(1- t)^2 K^2}{v^2} +2\langle\nabla f,\nabla F\rangle_u + (1+2Kt)\frac{K}{v}.
\end{equation}

Let $\varphi$ be some smooth cut-off function supported in $B_{1/2}$, satisfying $0\leq \varphi\leq 1$ on $B_{1/2}$ and $\varphi\equiv 1$ on $B_{1/4}$, so that 
\begin{equation}\label{eqn:cutoff}
	\abs{\triangle \varphi} +  \abs{\nabla \varphi}^2 \leq C
\end{equation}
on $B_{1/2}$ for some universal $C>0$.
By \eqref{eqn:Finequality} we obtain
\begin{eqnarray*}
	(\partial_t -\triangle_u) (\varphi^2 F) &=&  \varphi^2(\partial_t -\triangle_u )F -\triangle_u (\varphi^2) \cdot F -4 \langle \varphi \nabla \varphi, \nabla F\rangle_u \\
	&\leq&  -\varphi^2 F^2 -\frac{2(1-t)K}{v} (\varphi^2 F) - \frac{(1-t)^2K^2}{v^2}\varphi^2 \\
	&& + (1+ 2Kt)\frac{K}{v}\varphi^2 - \triangle_u (\varphi^2) \cdot F \\
	&& + 2\varphi^2 \langle\nabla f,\nabla F\rangle_u -4\langle \varphi \nabla \varphi,\nabla F\rangle_u .
\end{eqnarray*}

Next, we consider the maximum of $\varphi^2 F$ on $B_{1/2}\times (0,t_0]$. If it is smaller than $C_2+1$ for $C_2$ in \eqref{eqn:assumption}, then by the definition of $\varphi$, we get the desired bound of $F$ on $B_{1/4}\times (0,t_0]$ and finish the proof of Theorem \ref{thm:liyau}. Hence, we may assume it is no less than $C_2+1$. By \eqref{eqn:assumption} and \eqref{F1}, there is some $\tilde{t}\in (0,t_0]$ and $\tilde{\bx}\in B_{1/2}$ such that
\begin{equation}
	(\varphi^2 F) (\tilde{\bx},\tilde{t})= \max_{B_{1/2}\times (0,t_0]} \varphi^2 F \geq C_2+1.
	\label{eqn:tildet}
\end{equation}
At this point $(\tilde{\bx},\tilde{t})$, we have 
$(\partial_t -\triangle_u) (\varphi^2 F)\geq 0$ and $\grad(\varphi^2 F)=0$, and hence
\begin{eqnarray*}
	0&\leq &  -\varphi^2 F^2 -\frac{2(1-\tilde{t})K}{v} (\varphi^2 F) - \frac{(1-\tilde{t})^2K^2}{v^2}\varphi^2 \\
	&& + (1+2 K\tilde{t})\frac{K}{v}\varphi^2 - \triangle_u (\varphi^2) \cdot F \\
	&& -4\varphi F \langle\nabla f,\nabla \varphi\rangle_u + 8F \abs{\nabla \varphi}_u^2. 
\end{eqnarray*}
Moreover, since $v\geq 1$ on $B_{1/2}\times [0,t_0]$, 
by \eqref{eqn:assumption}, the bound
\eqref{eqn:cutoff} implies that $\abs{\triangle_u (\varphi^2)} + \abs{\nabla \varphi}_u^2\leq C$.
Thus, at $(\tilde{\bx},\tilde{t})$,  Young's inequality and \eqref{F2} give
\begin{eqnarray}\label{eqn:varphiF}
	0&\leq &  -\varphi^2 F^2 -\frac{2(1-\tilde{t})K}{v} (\varphi^2 F) - \frac{(1-\tilde{t})^2K^2}{v^2}\varphi^2 \\ \nonumber
	&& + (1+2 K\tilde{t})\frac{K}{v}\varphi^2 + C F  + \frac{1}{4}F \varphi^2 \abs{\nabla f}_u^2 \\ \nonumber
	&=& -\frac{3}{4} \varphi^2 F^2 + (-2+\frac{7}{4}\tilde{t})\frac{K}{v}(\varphi^2 F) - \frac{(1-\tilde{t})^2 K^2}{v^2} \varphi^2 \\ \nonumber
	&& + \frac{K+ 2K^2 \tilde{t}}{v}\varphi^2 + CF.
\end{eqnarray}
By (\ref{eqn:time}) and $\tilde{t}\leq t_0<1/2$, we have at $(\tilde{\bx},\tilde{t})$, 
\begin{equation}\label{eqn:smalltime}
	\frac{2K^2 \tilde{t}}{v}\varphi^2 \leq \frac{K^2}{4v^2} \varphi^2 \leq \frac{(1-\tilde{t})^2 K^2}{v^2}\varphi^2. 
\end{equation}
Using \eqref{eqn:smalltime} in \eqref{eqn:varphiF} and multiplying both sides by $\varphi^2$ yields that at $(\tilde{\bx},\tilde{t})$,
\begin{equation*}
	0\leq -\frac{3}{4} (\varphi^2 F)^2 + (-2+\frac{7}{4}\tilde{t})\frac{K}{v} \varphi^2 (\varphi^2 F) + \frac{K}{v} \varphi^4 + C (\varphi^2 F).
\end{equation*}
To proceed further, we must consider the sign of $K(\tilde{\bx},\tilde{t})$.
We claim that $K(\tilde{\bx},\tilde{t})>0$ since if we had $K(\tilde{\bx},\tilde{t})\leq 0$,
then by \eqref{F2} and \eqref{eqn:assumption} we would have
\begin{equation*}
	(\varphi^2 F)(\tilde{\bx},\tilde{t})\leq F(\tilde{\bx},\tilde{t})= \abs{\nabla f}_u^2(\tilde{\bx},\tilde{t}) + \frac{\tilde{t}K}{v}(\tilde{\bx},\tilde{t})\leq C_2,
\end{equation*}
which is a contradiction to \eqref{eqn:tildet}.
Therefore we may assume $K(\tilde{\bx},\tilde{t})>0$, which allows us to replace $\frac{K}{v}\varphi^4$ by a larger number $\frac{K}{v} \varphi^4 F$ (at $(\tilde{\bx},\tilde{t})$) because $(\varphi^2 F)(\tilde{\bx},\tilde{t})$, hence $F(\tilde{\bx},\tilde{t})$ is larger than $1$ as assumed in \eqref{eqn:tildet}. Precisely, at $(\tilde{\bx},\tilde{t})$,
\begin{eqnarray}\nonumber
	0&\leq& -\frac{3}{4} (\varphi^2 F)^2 + (-1+\frac{7}{4}\tilde{t})\frac{K}{v} (\varphi^4 F) +  C (\varphi^2 F)\\ \label{eqn:final}
	&\leq&  - \frac{3}{4} (\varphi^2 F)^2+ C(\varphi^2 F).
\end{eqnarray}
Here in the last line above, we used the assumption that $\tilde{t}\leq t_0<\frac{1}{2}$. Estimate \eqref{eqn:final} implies an upper bound for $(\varphi^2 F)(\tilde{\bx},\tilde{t})$, which gives the desired upper bound of $F$ on $B_{1/4}\times (0,t_0]$ and finishes the proof of Theorem \ref{thm:liyau}.

\subsection{Curvature upper bound in Theorem \ref{mainthm}}

By Theorem \ref{secondthm}, we know for sufficiently small $t>0$,
\begin{equation}
\label{vup}
	\max_{B_{1/2}} v_{cc}\leq \frac{2}{t}.
\end{equation}
%
%
Choose $t_1<1/2$ so small that the above holds on $[t_1,\frac{17}{16}t_1]$ and set
\begin{equation*}
v(\bx,t)= v_{cc}(\bx, t+t_1),
\end{equation*}
for $t\in [0,t_0]$, where $t_0:=\frac{t_1}{16}$.

We want to apply Theorem \ref{thm:liyau} to $v$. For this $v$, 
we check that the lower bound in \eqref{eqn:assumption} holds by Remark \ref{rmk:lower}, that the rest of \eqref{eqn:assumption} follows from Lemma \ref{lem:special} and that \eqref{eqn:time} holds because
$$\max_{B_{1/2}\times [0,t_0]} v
=\max_{B_{1/2}\times [t_1,\frac{17}{16}t_1]} v_{cc}
\leq \frac{2}{t_1}=\frac{1}{8t_0}$$
by \eqref{vup}.
Hence, there exists a constant $C_4$ depending only on $C_1$ (therefore, it is a universal constant) such that
\begin{equation*}
	K_{cc}(t)\leq \frac{C_4 v_{cc}}{t-t_1}
\end{equation*}
on $B_{1/4}\times [t_1,\frac{17}{16}t_1]$.
In particular, this implies the existence of another universal constant $C_5$ such that for $t=\frac{17}{16}t_1$,
\begin{equation*}
	\max_{B_{1/4}} K_{cc}(t)\leq \frac{C_5}{t^2}.
\end{equation*}
By the arbitrariness of $t_1$ (small as required above) and hence $t$, this proves the upper bound estimate inside $B_{1/4}$.
The upper bound over the whole of $B$ then follows from 
Proposition \ref{not_bad_at_infinity_prop}, which we prove in the next section.

\subsection{Good behaviour at spatial infinity}
\label{not_bad_at_infinity}

In this section we prove Proposition \ref{not_bad_at_infinity_prop}.
There are multiple ways one could prove this; for example, the asymptotics of Remark \ref{new_sandwich_rmk} allow one to argue using parabolic regularity theory that because
for each time $t$ 
the spatial asymptotics of $u_{cc}(t)$ agree with those of $h$ and $\tilde h$, we also have the curvatures agreeing. More precisely,
for each $L<\infty$, by taking $\ep>0$ small enough we can make $K_{cc}(t)$ as close as we like to $-1/(1+2t)$ on 
$B\backslash B_{1-\ep}\times [0,L]$, for example.

Instead of detailing this argument, we proceed via existing Ricci flow theory, and in particular the following result of B.-L. Chen.

\begin{prop}[Proposition 3.9 in \cite{BLChen}]
\label{BLChenProp}
	Let $g(t)$, $t\in [0,T]$, be a smooth solution to the Ricci flow 
on a two-dimensional manifold $M$, and let $x_0\in M$, $R>0$ and $v_0>0$. Assume $B_{g(t)}(x_0,R)$ is compactly contained in $M$ for every $t\in [0,T]$, and at $t=0$ that $\abs{K_{g(0)}}\leq R^{-2}$ on $B_{g(0)}(x_0,R)$ and $\Vol_{g(0)}(B_{g(0)}(x_0,R))\geq v_0 R^2$.
Then there exists a constant $\eta>0$, depending only on $v_0$, such that
for $0\leq t\leq \min \set{T,\eta R^2}$, we have
	\begin{equation*}
		\abs{K_{g(t)}}\leq 2R^{-2}\qquad\text{ on }B_{g(t)}(x_0,\frac{R}{2}).
	\end{equation*}
\end{prop}
Here $B_g(x_0,R)$ is the geodesic ball centred at $x_0$ with radius $R$ measured with respect to the metric $g$.

For $\ep>0$ as given in Proposition \ref{not_bad_at_infinity_prop}, choose $R\in (0,1/2]$ as large as possible so that for each $x_0\in B\setminus B_{\ep/2}$, we have 
$R\leq \half\inj_h(x_0)$.
For each $x_0\in B\setminus B_{\ep/2}$, we can then apply Proposition \ref{BLChenProp} with $M=B$, $v_0$ equal to the area of the unit disc in the flat plane,
and $g(t)=g_{cc}(t+\ga)$ for $\ga>0$ sufficiently small so that $B_{g(0)}(x_0,R)$ 
equipped with the metric $g(0)$ is sufficiently close to a ball in hyperbolic space of radius $R$ that
the hypotheses of the proposition are satisfied.

The proposition, in the limit $\ga\downto 0$, gives us a curvature bound at $x_0$ depending only on $\ep$, which holds for a time interval also depending only on $\ep$.
We may then invoke Shi's local derivative bounds to obtain control on all space and time derivatives of the curvature, depending only on the order of the derivative and $\ep$, and not on $x_0$. Since the Gauss curvature starts initially at $-1$, the proposition follows.

\section{Curvature lower bound in Theorem \ref{mainthm}}\label{sec:lower}


In this section, we prove the lower curvature  bound in Theorem \ref{mainthm}.  By Theorem \ref{secondthm}, for any $\mu\in (0,\frac{1}{8})$ and $t>0$ sufficiently small, depending on $\mu$, we have
\begin{equation}
\label{vcc_lower}
	v_{cc}(0,t)\geq \frac{\mu}{t}.
\end{equation}
This means that when $t$ is small, $u_{cc}(0,t)$ is very large. On the other hand, 
by Remark \ref{new_sandwich_rmk}, we have
\begin{equation}\label{eqn:belowcusp}
	u_{cc}(\bx,t)\leq (1+2t)h(\bx)=(1+2t) \frac{1}{r^2(\log r)^2},
\end{equation}
throughout $B\setminus\set{0}$, for 
all $t\geq 0$.
These two facts combined together will imply the existence of some large positive curvature of $g_{cc}(t)$. The proof needs another family of special metrics lying below, but touching, the (scaled) hyperbolic cusp metric, that we now construct.
Consider the  family of metrics $u_{\be,K}(dx^2+dy^2)$ on $\R^2$ parametrized by $K>0$ and $\beta>0$, where
\begin{equation*}
u_{\be,K}(r):=\frac{\beta^2}{\left( 1+\frac{\beta^2K r^2}{4} \right)^2}.
\end{equation*}
Each of these is the metric of a (punctured) sphere with constant curvature $K$  parametrized so that the conformal factor at the origin is $\beta^2$. The next lemma gives us the touching family mentioned above.

\begin{lemma}
\label{ubetaK0_lemma}
\label{lem:touch2}
For each fixed $\alpha>1$, there exist a continuous strictly increasing function $K_0: (\alpha e ,\infty) \to (0,\infty)$ and a continuous strictly decreasing function $r_0:(\alpha e,\infty)\to (0,e^{-1})$ such that

(1)
for any $\beta> \alpha e$, we have 
\begin{equation}\label{eqn:touch}
	u_{\beta,K_0(\beta)}(r)=\frac{\beta^2}{\left( 1+\frac{\beta^2K_0(\beta) r^2}{4} \right)^2} \leq \frac{\alpha^2}{r^2(\log r)^2}=\al^2 h(r)
\end{equation}
for all $r\in (0,1)$ with equality only at $r_0(\beta)$.

(2) the asymptotic behavior of $K_0$ and $r_0$ when $\beta$ approaches $\alpha e$ or $\infty$ is given by
\begin{equation}\label{eqn:asymptotics}
	\begin{aligned}
	\lim_{\beta\to \infty} K_0=\infty &\qquad &  \lim_{\beta\downto \alpha e} K_0=0 \\
	\lim_{\beta\to \infty} r_0=0 &\qquad &  \lim_{\beta\downto \alpha e} r_0 = e^{-1}.
	\end{aligned}
\end{equation}
Moreover, we have
\beq \label{eqn:ubetaK0}
\lim_{\beta\to \infty} \inf_{r\in [0,r_0]}u_{\be,K_0}(r)\to\infty.
\eeq

(3) we have the following lower bound for $K_0$: 
\begin{equation}
\label{K0_lower}
	K_0\geq \frac{2}{\alpha} \left( (\log\frac{\beta}{2\alpha})^2 -1 \right).
\end{equation}
\end{lemma}

\begin{proof}
The proof consists of two steps. First, we prove the existence of some $K_0$ and $r_0$ satisfying (1). Then in the second step, we show that (2) and (3) also hold for this $K_0$ and $r_0$ using some results from the first step.

We start the first step by giving equivalent forms of (1). It is elementary that (1) is equivalent to
\vskip 3mm
($1'$)
for any $\beta> \alpha e$, we have 
\begin{equation}
\label{simpler}
\frac{\beta}{\alpha} r (-\log r) \leq 1+ \frac{\beta^2 K_0(\beta) r^2}{4},
\end{equation}
for all $r\in (0,1)$ with equality only at $r_0(\beta)$.
\vskip 3mm

We claim that ($1'$) and hence (1) is also equivalent to

\vskip 3mm
($1''$) for any $\beta>\alpha e$, we have
\begin{eqnarray}\label{1}
		\frac{\beta}{\alpha} r_0 (-\log r_0) &=&  1+ \frac{\beta^2 K_0 r^2_0}{4} \\
		\frac{1}{\alpha} (-\log r_0 -1) &=&  \frac{1}{2}\beta K_0 r_0.
		\label{2}
	\end{eqnarray}
\vskip 3mm
	It is easier to see that ($1''$) is a necessary condition of ($1'$) because \eqref{1} and \eqref{2} are nothing but the claim that both sides of \eqref{simpler} and their first order derivatives with respect to $r$ agree at $r_0(\beta)$. To see that it is also sufficient, we observe that the left-hand side of \eqref{simpler} is a strictly concave function of $r$ on $(0,1)$, while the right-hand side is strictly convex.

	With the equivalence of (1) and ($1''$) in mind, it suffices to find $K_0(\beta)$ and $r_0(\beta)$ satisfying \eqref{1} and \eqref{2}.  While solving $K_0(\beta)$ and $r_0(\beta)$ from \eqref{1} and \eqref{2} seems not easy, we can obtain explicit formulae relating $\beta$ and $K_0$ to $r_0$. More precisely, we eliminate $K_0$ to get
\begin{equation}
	\beta=\frac{2\alpha}{r_0(-\log r_0+1)}.
	\label{eqn:betar0}
\end{equation}
Substituting \eqref{eqn:betar0} into \eqref{2} yields
\begin{equation}
	K_0=\frac{2}{\alpha}(-\log r_0-1)(-\log r_0+1).	
	\label{eqn:K0r0}
\end{equation}
It is elementrary to check that $\beta$ as a function of $r_0$ given in \eqref{eqn:betar0} is a decreasing diffeomorphism from $(0,e^{-1})$ to $(\alpha e, \infty)$. Therefore, it is equivalent to say that \eqref{eqn:betar0} defines a function $r_0(\beta)$ which is a decreasing diffeomorphism from $(\alpha e,\infty)$ to $(0,e^{-1})$. $K_0(r_0)$ as given in \eqref{eqn:K0r0} is a decreasing diffeomorphism from $(0,e^{-1})$ to $(0,\infty)$, which we compose with the $r_0(\beta)$ just obtained to get a function $K_0(\beta)$ that is an increasing diffeomorphism from $(\alpha e,\infty)$ to $(0,\infty)$. The $K_0(\beta)$ and $r_0(\beta)$ thus obtained satisfy ($1''$) and hence (1), finishing the first step of the proof.

For (2), we notice that the asymptotic behavior of $K_0$ and $r_0$ as in \eqref{eqn:asymptotics} is proved in the previous paragraph. By the monotonicity of $u_{\beta,K_0(\beta)}(r)$ as a function of $r$ and \eqref{1}, we have
\begin{equation*}
	\inf_{r\in [0,r_0(\beta)]}u_{\beta,K_0(\beta)}(r)= u_{\beta,K_0(\beta)}(r_0(\beta))= \frac{\alpha^2}{(r_0(\beta)^2 (\log r_0(\beta))^2)}\to \infty
\end{equation*}
when $\beta\to \infty$, which is \eqref{eqn:ubetaK0}.

An easy observation from \eqref{eqn:betar0} is that
\begin{equation*}
	\beta=\frac{2\alpha}{r_0(-\log r_0+1)} \leq \frac{2\alpha}{r_0},
\end{equation*}
which gives $r_0\leq \frac{2\alpha}{\beta}$. By the monotonicity of \eqref{eqn:K0r0}, we obtain \eqref{K0_lower}.
\end{proof}

Now, we return to the proof of Theorem \ref{mainthm}. For any $c_1$ larger than $32$ as in Theorem \ref{mainthm}, we  choose any 
$\mu<\frac{1}{8}$ and any $\al>1$ so that 
\begin{equation}
\label{choice_of_constants}
\frac{1}{32}>\frac{2 \mu^2}{\al}> \frac{1}{c_1}.
\end{equation}
Then we can pick $t_0>0$ such that for all $0<t<t_0$,
we have
\begin{equation}
	1+2t<\alpha^2,
	\label{eqn:talpha}
\end{equation}
\begin{equation}
\label{eqn:higher}
v_{cc}(0,t)\geq \frac{1}{8t}+\frac{1}{2}(1+\log 4t)> \frac{\mu}{t},
\end{equation}
and
\begin{equation}
	\frac{2}{\alpha}\left[ \left(\frac{\mu}{t} -\log 2\alpha\right)^2 -1 \right]>\frac{1}{c_1 t^2}.
	\label{eqn:smallt}
\end{equation}
where we have used Theorem \ref{secondthm} in \eqref{eqn:higher}.
We claim that 
\begin{equation*}
	\max_B K_{cc}(t)>\frac{1}{c_1 t^2},
\end{equation*}
for all $t\in (0,t_0)$,
which would conclude the proof of Theorem \ref{mainthm}.

To see the claim is true for a given $t\in (0,t_0)$, 
consider the  family $w_\beta$ of functions defined for 
$\beta>\alpha e$ by
\begin{equation*}
	w_\beta(r)=\left\{
		\begin{array}[]{ll}
			u_{\beta,K_0}(r) & \quad 0\leq r<r_0\\
			\alpha^2 h(r) & \quad  r_0\leq r<1,
		\end{array}
		\right.
\end{equation*}
where $K_0$ and $r_0$ are given in Lemma \ref{lem:touch2}.
Each value $w_\be(r)$ will vary continuously in $\be$ by 
Lemma \ref{lem:touch2}.
%
%
By \eqref{eqn:talpha} and the fact that $u_{cc}(t)\leq (1+2t) h$,
by \eqref{eqn:belowcusp}, we have  
\begin{equation}\label{eqn:smaller}
	u_{cc}(t)< \alpha^2 h.
\end{equation}
Therefore, by construction of $w_\be$, and by \eqref{eqn:ubetaK0} of Lemma \ref{ubetaK0_lemma}, we have
$w_\beta(r)> u_{cc}(r,t)$ for large enough $\beta$. 

We now reduce $\be$ from such a large value until the largest $\be$ for which this fails, i.e. so that $w_\beta(r)\geq  u_{cc}(r,t)$, with equality for some $r_1\in [0,1)$.
By the definition of $w_\be$ and  \eqref{eqn:higher}
we then have
$$\be^2=w_\beta(0)\geq u_{cc}(0,t)\geq e^{\frac{2\mu}{t}},$$
and in particular, 
\begin{equation}
	\log \beta\geq \frac{\mu}{t}. 
	\label{eqn:betalarge}
\end{equation}

By \eqref{eqn:smaller} and the definition of $w_\beta$, we know $r_1<r_0$, where $r_0=r_0(\beta)$ is given in Lemma \ref{lem:touch2}. Because $u_{cc}(t)$ and $w_\beta(r)=u_{\beta,K_0}(r)$ are two smooth functions in a small neighbourhood of $r_1$ and $u_{cc}$ touches $u_{\beta,K_0}$ from below at $r_1$, we deduce that
$K_{cc}(r_1,t)\geq K_0$. Finally, we use (3) of Lemma \ref{lem:touch2}, \eqref{eqn:betalarge} and \eqref{eqn:smallt} to conclude that
\begin{equation*}
	\max_B K_{cc}(t)\geq K_0 \geq \frac{2}{\alpha}\left[ (\log\beta -\log 2\alpha)^2-1 \right] >\frac{1}{c_1 t^2},
\end{equation*}
for each $t\in (0,t_0)$, completing the proof of Theorem \ref{mainthm}.

\vskip 1cm

\emph{Acknowledgements:} The first author was supported by EPSRC Programme grant number EP/K00865X/1 and the second author was supported by NSFC 11471300.

\appendix
\section{Appendix: Proofs of Lemmas}
\subsection{Proof of Lemma \ref{lem:touch}}
\begin{proof}
	By (\ref{eqn:delta}) and (\ref{eqn:varepsilon}), it suffices to show that
\begin{equation*}
	r^2 (\log r)^2 \leq (-\log r_0) \left( r_0^2 + (-\log r_0-1) r^2 \right),
\end{equation*}
or equivalently that 
$F(r):=r^2 (\log r)^2 - (-\log r_0) \left( r_0^2 + (-\log r_0-1) r^2 \right)\leq 0$,
with equality if and only if $r=r_0$. Equality at $r=r_0$ is clear. 
We compute
$$F'(r)=\left[2r((-\log r)+(-\log r_0)-1)\right]((-\log r)-(-\log r_0)),$$
and because $0<r_0<1/e$, the part in square brackets is positive, and we see that
$F'(r)<0$ for $0<r<r_0$, while $F'(r)>0$ for $r_0<r<1$, which is enough to conclude
that $F(r)<0$ for $0<r<r_0$ and $r_0<r<1$.
\end{proof}

\subsection{Proof of Lemma \ref{lem:increase}}
\begin{proof}
When $r=\abs{\bx}>r_0$, $u_{r_0}(\bx)=h(\bx)$ does not depend on $r_0$ at all, and is hence trivially decreasing. For $\abs{\bx}<r_0$, we use \eqref{eqn:ur0}, \eqref{eqn:delta} and \eqref{eqn:varepsilon} to compute
\begin{eqnarray*}
	\frac{\partial}{\partial r_0} (u_{r_0}^{-1})&=& \frac{\partial}{\partial r_0}\left( -r_0^2 \log r_0 + r^2 (\log r_0)^2 + r^2 \log r_0 \right) \\
	&=& -2r_0\log r_0 -r_0 + \frac{r^2}{r_0} (2\log r_0) + \frac{r^2}{r_0} \\
	&=& \frac{1}{r_0}\left( -2 r_0^2 \log r_0 -r_0^2 + 2 r^2 \log r_0 +r^2 \right) \\
	&=& \frac{1}{r_0} (r_0^2-r^2) (-2\log r_0 -1),
\end{eqnarray*}
which is positive since $r_0<\frac{1}{e}$.
\end{proof}

\subsection{Proof of Lemma \ref{lem:Fn}}
\begin{proof}
To prove \eqref{eqn:special}, we compute
\begin{equation*}
\frac{1}{2}\log h = - \log \abs{r\log r}
\end{equation*}
and
\begin{equation*}
\begin{aligned}
	h^{-1} \abs{\nabla \log \left( \frac{1}{2}\log h \right)}^2 &=   \frac{r^2(\log r)^2 }{ (\log \abs{r\log r})^2} \frac{(\log r +1)^2}{\abs{r\log r}^2}\\
	&= \frac{(\log r +1)^2 }{ (\log \abs{r\log r})^2}.
\end{aligned}
\end{equation*}
It is not hard to see that the limit of the above quantity is $1$ as $r\downto 0$ and $0$ as $r\upto 1$. It is therefore bounded, by continuity, as required for \eqref{eqn:special}.

Because $u_n=h$ for $r\in [r_n,1)$, we see that \eqref{Fn_ineq} holds for this range of values of $r$, by virtue of \eqref{eqn:special}.

Having dealt with the hyperbolic cusp part, i.e. for $r\in [r_n,1)$, it 
remains to verify \eqref{Fn_ineq} for $r< r_n$, i.e. on the cigar part where
\begin{equation*}
	u_n=\frac{\varepsilon}{\delta+r^2}
\end{equation*}
and hence
\begin{equation*}
	v_n=\frac{1}{2}\log u_n=\frac{1}{2}\left( \log \varepsilon -\log(\delta+r^2) \right).
\end{equation*}
It suffices then to show that $F_n$ is an increasing function of $r\in (0,r_n]$, since we have already established the bound for $r=r_n$.
We compute
\begin{equation*}
	\abs{\nabla f_n}^2 =\frac{\abs{\nabla v_n}^2}{v_n^2} = \frac{r^2}{(\delta+r^2)^2}\frac{4}{(\log \varepsilon-\log(\delta+r^2))^2}
\end{equation*}
and thus
\begin{equation}\label{eqn:Fn}
	F_n=u_n^{-1} \abs{\nabla f_n}^2 =\frac{1}{\varepsilon} \frac{1}{(1+\frac{\delta}{r^2})}\frac{4}{(\log \varepsilon-\log(\delta+r^2))^2}.
\end{equation}
By 
\eqref{eqn:varepsilon},
we have
\begin{equation*}
	\delta+r^2\leq \delta+r_n^2 =\ep r_n^2(-\log r_n)^2\leq \frac{\ep}{e^2}< \varepsilon,
\end{equation*}
which together with \eqref{eqn:Fn} implies that $F_n$ is an increasing function of $r$, as required. 
\end{proof}

{\sc mathematics institute, university of warwick, coventry, CV4 7AL,
uk}\\
\url{http://www.warwick.ac.uk/~maseq}

{\sc School of mathematical sciences, university of science and technology of China, Hefei, 230026, China}

\end{document}